\documentclass[a4paper]{article}

\usepackage{amsfonts}
\usepackage{amsmath}
\usepackage{mathrsfs}
\usepackage{amssymb,color}

\usepackage[latin1]{inputenc}
\usepackage{amsthm}
\usepackage{graphicx}		
\usepackage{enumerate}
\newtheorem{theorem}{Theorem}[section]

\newtheorem{lemma}[theorem]{Lemma}
\newtheorem{remark}[theorem]{Remark}

\newtheorem{exemplo}[theorem]{Example}

\usepackage{authblk} 					

	\title{Existence of solutions for nonlinear $p-$Laplacian difference equations.}
	\date{}
	\author[1]{\textbf{Lorena Saavedra}\footnote{Corresponding author. Email: lorena.saavedra@usc.es}}
	\author[2,3]{\textbf{Stepan Tersian}\footnote{Email: sterzian@uni-ruse.bg}
		}

	\affil[1]{Departamento de An\'alise Matem\'atica, Facultade de Matem\'aticas, Universidade de Santiago de Compostela, Santiago de Compostela, Galicia,
		Spain.}
	\affil[2]{Department of Mathematics, University of Ruse, Ruse, Bulgaria. }
	\affil[3]{Associate at: Institute of Mathematics and Informatics, Bulgarian Academy of Sciences, 1113- Sofia, Bulgaria.}
\begin{document}

	\maketitle
	\begin{abstract}		
		The aim of this paper is the study of existence of solutions for nonlinear $p-$Laplacian difference equations. In the first part, the existence of a nontrivial homoclinic solution for a discrete $p-$Laplacian problem is proved. The proof is based on the mountain-pass theorem of Brezis and Nirenberg.	Then, we study the existence of multiple solutions for a discrete $p-$Laplacian boundary value problem. In this case the proof is based on the theorem of D. Averna and G. Bonanno, which ensures the existence of three critical points for a suitable functional.
	\end{abstract}
	
	\section{Introduction}
This paper is divided in two parts. Both of them are devoted to study the existence of one or multiple solution of problems with $p-$Laplacian difference equations:
\[\Delta^2\left( \varphi_{p_2}\left( \Delta^2u(k-2)\right) \right) -a\,\Delta\left( \varphi_{p_1}\left( \Delta u(k-1)\right) \right) +V(k)\,\varphi_q\left( u(k)\right) =\lambda f(k,u(k))\,.\]

The first part is based on the mountain-pass theorem of Brezis and Nirenberg \cite{BreNir}. Following the steps as \cite{CaLiTe}, we obtain the existence of an homoclinic solution for the given equation.

In the second part, we obtain the existence of at least three solutions for the difference equation with $p_1=p_2=q$ and the Dirichlet boundary conditions, by generalizing a result given in \cite{CanGio} for the problem
\[-\Delta \left( \varphi_p\left( \Delta u(k-1)\right) \right) =\lambda\,f(k,u(k))\,,\ k\in[1,T]\,,\quad u(0)=u(T+1)=0\,.\]

 Such result is obtained by applying \cite[Theorem 2.1]{AvBo} to our boundary value problem.

In both cases, we show how our arguments should be modified in order to obtain the analogous result for  higher order problems.

The study of $p-$Laplacian difference equations has been developed in the literature. In addition to the previously mentioned (\cite{CaLiTe}, \cite{CanGio}), we refer \cite{He}, where the following problem is studied:
\[\Delta \left( \varphi_p\left( \Delta u(k-1)\right) \right)+a(t)\,f(k,u(k))=0\,,\ k\in[1,T+1]\,,\quad \Delta u(0)=u(T+2)=0\,,\]
where $a(t)$ is a is a positive function. Moreover, in \cite{WaGua}, the existence of three positive solutions of this problems is studied. 

Recently, in \cite{Dim}, it is proved the existence of at least three solutions of the problem
\begin{eqnarray}
\nonumber \Delta^2\left( \varphi_{p}\left( \Delta^2u(k-2)\right) \right) +\alpha \,\varphi_p\left( u(k)\right)& =&\lambda f(k,u(k))\,,\quad k\in[1,T]\,,\\\nonumber
u(0)=\Delta u(-1)=\Delta^2 u(T)&=&0\,,\\\nonumber
\Delta \left( \varphi_p\left( \Delta^2 u\left( T-1\right) \right) \right)& =&\mu g(u(T+1))\,,\end{eqnarray}
	 where $\alpha$, $\lambda$ and $\mu$ are real parameters, $f$ and $g$ are continuous.
	 
	 Moreover, we refer \cite{IaTe},\cite{MiRaTe} and \cite{TaLiXi}, where the existence of homoclinic solution for different discrete second order problems is studied.
	 
	 This paper is structured in two parts. We introduce the considered problems in the beginning of these parts. Then, we construct the related variational formulation. After some preliminaries, the existence solutions results are proved. Finally, we give  examples to the results obtained.
	 \section{Homoclinic Solutions}
	 This section is focused on the study of the existence of homoclinic solutions for the following problem:
	 
	 \begin{eqnarray}
	 	\nonumber
	 	\Delta^2\left( \varphi_{p_2}\left( \Delta^2u(k-2)\right) \right) -a\,\Delta\left( \varphi_{p_1}\left( \Delta u(k-1)\right) \right) +V(k)\,\varphi_q\left( u(k)\right) &=&\lambda f(k,u(k))\,,\\\label{P::h1}\\\nonumber
	 	\lim_{\left| k\right| \rightarrow +\infty}\left| u(k)\right| &=&0\,,
	 	\end{eqnarray}
	 where $a>0$ is fixed, $p_i\geq q>1$ for $i=1,2$ and \begin{eqnarray}\label{Ec::D1}\Delta u(k)&=&u(k+1)-u(k)\,,\\\label{Ec::D2} \Delta^i u(k)&=&\Delta^{i-1} u(k+1)-\Delta^{i-1} u(k)\,,\ \text{if } i\geq 2\end{eqnarray} are the difference operators.
	 
	 For each $p>1$:  
	 \begin{eqnarray}
	 \label{Ec::Phi} 	\varphi_p\colon\mathbb{R} &\longrightarrow&\mathbb{R}
	 \\\nonumber t&\longmapsto&\varphi_p(t)=t\,\left| t\right| ^{p-2}\,,
	 \end{eqnarray}and $V\colon \mathbb{Z}\rightarrow \mathbb{R}$ is a $T-$periodic positive function for $T$ a fixed integer. Let us denote 
	 \[0<V_0=\min\{V(0),\dots,V(T-1)\}\,,\quad \text{and}\ V_1=\max\{V(0),\dots,V(T-1)\}\,.\]
	 
	 Moreover, $f\colon \mathbb{Z}\times \mathbb{R}\rightarrow \mathbb{R}$ verifies the following assumptions:
	 \begin{itemize}
	 	\item[$(F_1)$] $f(k,\cdot)\colon \mathbb{R}\rightarrow \mathbb{R}$ is a continuous function and $f(\cdot,t)\colon\mathbb{Z}\rightarrow \mathbb{R}$ is a $T-$periodic function.
	 	\item[$(F_2)$] The potential function,
	 	$F(k,t)=\int_0^tf(k,t)\,dt$,
	 	satisfies the Rabinowitz's type condition:
	 	
	 	There exist $\mu\in \mathbb{R}$, such that $\mu> p_i\geq q>1$, for $i=1,2$ and $s>0$, such that
	 	\begin{eqnarray}
	 	\nonumber
	 	\mu\,F(k,t)\leq t\,f(k,t)\,, && \forall k\in\mathbb{Z}\,,\  t\neq0\,,\\\nonumber
	 	F(k,t)>0\,,&& \forall k\in\mathbb{Z}\,,\ \forall t\geq s>0\,.
	 	\end{eqnarray}
	 	\item[$(F_3)$]$f(k,t)=o\left( \left| t\right| ^{q-1}\right)$ as $\left| t\right| \rightarrow 0$. 
	 \end{itemize}
	 
	 Even though we are going to focus on problem \eqref{P::h1}, we are going to indicate in each step how the results can be modified in order to obtain the result of existence of homoclinic solution for the following problem 	 
	 \begin{eqnarray}
	 	\nonumber
	 	\Delta^n\left( \varphi_{p_n}\left( \Delta^nu(k-2)\right) \right) +\sum_{i=1}^{n-1}(-1)^ia_i\,\Delta^{n-i}\left( \varphi_{p_{n-i}}\left( \Delta^{n-i} u(k-1)\right) \right)&& \\\nonumber
	 	+(-1)^nV(k)\,\varphi_q\left( u(k)\right) +(-1)^{n+1}\lambda f(k,u(k))&=&0\,,\\	\label{P::h2}\\\nonumber
	 	\lim_{\left| k\right| \rightarrow +\infty}\left| u(k)\right| &=&0\,,
	 	\end{eqnarray}
	 where $V$, $f$, $\varphi_p$ and $\Delta^j$ have been previously introduced, $\mu>p_i\geq q>1$ and $a_i\geq 0$ for all $j\in\{1,2,\dots,n\}$, with $a_n=1$.
	 
	Define $\phi_p(t)=\dfrac{\left| t\right| ^p}{p}$. It is trivial that $\phi_p'(t)=\varphi_p(t)$ for every $p>1$.
	
	Let $\ell^q=\{\left( u(k)\right) _{k\in\mathbb{Z}}\ \mid \sum_{k\in\mathbb{Z}}\left| u(k)\right| ^q<\infty\}$ be the Banach space with the norm $|u|_q^q=\sum_{k\in\mathbb{Z}}|u(k)|^q$ and $
	J\colon \ell^q\longrightarrow \mathbb{R}$ the functional $J(u)=\Phi(u)-\lambda\sum_{k\in \mathbb{Z}} F(k,u(k))$, where
	\[\Phi(u)=\sum_{k\in\mathbb{Z}}\left( \phi_{p_2}\left( \Delta^2\,u(k-2)\right) +a\,\phi_{p_1}\left(\Delta\,u(k-1)\right) +V(k)\,\phi_q(u(k))\right) \,.\]
	
	\begin{remark}
		In the case of problem \eqref{P::h2}, we consider
		\[	\Phi(u)=\sum_{k\in\mathbb{Z}}\left( \phi_{p_n}\left( \Delta^n\,u(k-n)\right) +\sum_{i=1}^{n-1}a_i\,\phi_{p_{n-i}}\left(\Delta^{n-i}\,u(k-(n-i))\right) +V(k)\,\phi_q(u(k))\right) \,.\]
	\end{remark}
	
	We have the following result:
	
	\begin{lemma}\label{L::2}
		The functional $J\colon \ell^q\rightarrow \mathbb{R}$ is well defined and $C^1-$differentiable. 
		
		Moreover, its critical points are solutions of \eqref{P::h1}.
	\end{lemma}
	\begin{proof}
		Let us see first that it is well-defined. In order to do that, consider the following inequality:
		\begin{equation}\label{Ec::Inbc}\left( \dfrac{x+y}{2}\right) ^p\leq \frac{x^p+y^p}{2}\,,\end{equation} which is verified for every non negative $x$, $y$ and $p>1$.
		
		Applying this inequality for $p_2>1$ twice, we have:
		\begin{eqnarray}\nonumber \sum_{k\in\mathbb{Z}}\phi_{p_2}\left( \Delta^2\left( u(k-2)\right) \right) &\leq& \dfrac{1}{p_2}\sum_{k\in\mathbb{Z}}2^{p_2-1}\left( \left|\Delta u(k-1)\right|^{p_2}+\left|\Delta u(k-2)\right|^{p_2}\right) \\\nonumber &\leq& \dfrac{1}{p_2}\sum_{k\in\mathbb{Z}} 4^{p_2-1}\left( \left|u(k)\right|^{p_2}+2\left| u(k-1)\right|^{p_2} +\left| u(k-2)\right|^{p_2}\right)\\\nonumber &=&\frac{4^{p_2}}{p_2}  \sum_{k\in\mathbb{Z}} \left| u(k)\right| ^{p_2}\,.\end{eqnarray}
		
		Now, since $p_2\geq q$, it is well-known that $\ell^{q}\subset \ell^{p_2}$. Hence, using the fact that $u\in \ell^q$, we conclude that 
		\[\sum_{k\in\mathbb{Z}}\phi_{p_2}\left( \Delta^2\left( u(k-2)\right) \right)\leq \frac{4^{p_2}}{p_2}  \sum_{k\in\mathbb{Z}} \left| u(k)\right| ^{p_2}<+\infty\,.\]
		
		Now, let us apply inequality \eqref{Ec::Inbc} for $p_1$. We have, taking into account  that $u\in \ell^q\subset \ell^{p_1}$:
		
				\begin{eqnarray}\nonumber \sum_{k\in\mathbb{Z}} \left|\Delta u(k-1)\right|^{p_1} &\leq& \dfrac{1}{p_1}\sum_{k\in\mathbb{Z}} 2^{p_1-1}\left( \left|u(k)\right|^{p_1}+\left| u(k-1)\right|^{p_1}\right)\\\nonumber &=&\frac{2^{p_1}}{p_1}  \sum_{k\in\mathbb{Z}} \left| u(k)\right| ^{p_1}<+\infty\,.\end{eqnarray}
				
				Moreover, we have
				\[\sum_{k\in\mathbb{Z}}V(k)\left| u(k)\right| ^q\leq V_1\sum_{k\in\mathbb{Z}}\left| u(k)\right| ^q<+\infty\,.\]
				
				Finally, for all $\delta\in(0,1)$, there exists $N>0$ sufficiently large such that $\left| u(k)\right|^q<\delta<1$ if $\left| k\right| >N$. Moreover, under the assumption ($F_3$), we have
				\[\exists \delta\in(0,1)\ \text{such that } F(k,u(k))<\left| u(k)\right| ^q<\delta <1\,, \left| k\right| >N\,.\]
				Thus, $\sum_{k\in\mathbb{Z}}F(k,u(k))<+\infty$ and $J$ is a well defined functional in $\ell^q$.
				
			For all $v\in \ell^q$, we have:			
			\begin{eqnarray}\nonumber
			\left\langle J'(u),v\right\rangle &=& \sum_{k\in\mathbb{Z}}\left( \varphi_{p_2}\left( \Delta^2u(k-2)\right) \,\Delta^2v(k-2)+a\,\varphi_{p_2}\left( \Delta u(k-1)\right) \Delta v(k-1)\right) \\\nonumber &&+\sum_{k\in\mathbb{Z}} V(k)\,\varphi_q(u(k))v(k)-\lambda\sum_{k\in\mathbb{Z}} f(k,u(k))\,v(k) \,.
			\end{eqnarray}
			
			Having into account that,
			\begin{eqnarray}\nonumber
			\sum_{k\in\mathbb{Z}} \varphi_{p_2}\left( \Delta^2u(k-2)\right) \,\Delta^2v(k-2)&=&	\sum_{k\in\mathbb{Z}} \varphi_{p_2}\left( \Delta^2u(k-2)\right) \,\left( v(k)-2v(k-1)+v(k-2)\right) \\\nonumber
			&=&	\sum_{k\in\mathbb{Z}}\Delta^2\left(  \varphi_{p_2}\left( \Delta^2u(k-2)\right)\right)  \,v(k)\,,\\\nonumber\\\nonumber
			\sum_{k\in\mathbb{Z}}\varphi_{p_2}\left( \Delta u(k-1)\right) \Delta v(k-1)&=&\sum_{k\in\mathbb{Z}}\varphi_{p_2}\left( \Delta u(k-1)\right) \left( v(k)- v(k-1)\right) \,,\\\nonumber
			&=&-\sum_{k\in\mathbb{Z}}\Delta\left( \varphi_{p_2}\left( \Delta u(k-1)\right)\right)v(k)\,,
			\end{eqnarray}
			we conclude that for all $v\in \ell^q$,
				\begin{eqnarray}\nonumber
				\left\langle J'(u),v\right\rangle &=& \sum_{k\in\mathbb{Z}}\left(\Delta^2\left( \varphi_{p_2}\left( \Delta^2u(k-2)\right)\right) -a\,\Delta\left( \varphi_{p_2}\left( \Delta u(k-1)\right)\right) \right)  v(k) \\\nonumber &&+\sum_{k\in\mathbb{Z}} V(k)\,\varphi_q(u(k))v(k)-\lambda\sum_{k\in\mathbb{Z}} f(k,u(k))\,v(k) \,.
				\end{eqnarray}
				
				So, we can obtain the partial derivatives as follows:
				\[\dfrac{\partial J(u)}{\partial u(k)}=\Delta^2\left( \varphi_{p_2}\left( \Delta^2u(k-2)\right)\right) -a\,\Delta\left( \varphi_{p_2}\left( \Delta u(k-1)\right)\right)+V(k)\,\varphi_q(u(k))-\lambda f(k,u(k))\,,\]
				which are continuous functions. 
				
				Moreover, from this we conclude that the critical points of $J$ are the solutions of \eqref{P::h1}.
	\end{proof}
	
	\begin{remark}
	In order to prove Lemma \ref{L::2} for the problem \eqref{P::h2}, we obtain using an induction argument that for each $i=0,\dots,n-1$, the following inequality is fulfilled:
	\[\sum_{k\in\mathbb{Z}}\phi_{p_{n-i}}\left( \Delta^{n-i}u(k-(n-i))\right) \leq \dfrac{2^{n-i}}{p_{n-i}}\sum_{k\in\mathbb{Z}}\left|u(k)\right| ^{p_{n-i}}\,,\]
	hence, since $p_{n-i}\geq q$, we have that $u\in\ell^q\subset\ell^{p_{n-i}}$. Thus, we can conclude that $J$ is a well-defined functional.
	
	Moreover, also using an induction argument, we can prove that for all $v\in\ell^q$:
\[\sum_{k\in\mathbb{Z}}\varphi_{p_{i}}\left( \Delta^{i}u(k-i)\right) \Delta^{i}v(k-i)=(-1)^{i}\sum_{k\in\mathbb{Z}}\Delta^{i}\left(\varphi_{i}\left( \Delta^{i}u(k-i)\right) \right) v(k)\,,\]
which allows us to prove that every critical point of $J$ is a solution of problem \eqref{P::h2} with the same arguments as in Lemma \ref{L::2}.
	 
	\end{remark}
	
	Now, let us introduce the mountain-pass theorem of Brezis and Nirenberg \cite{BreNir}, which we are going to use in order to obtain the homoclinic solutions of \eqref{P::h1} and \eqref{P::h2}.
	
	Let $X$ be a Banach space with norm $\left\| \cdot \right\| $ and $I\colon X\rightarrow \mathbb{R}$ be a $C^1-$functional. 
	
	We say that $I$ satisfies the $(PS)_c$ condition if every sequence $(x_k)\subset X$ such that 
	\begin{equation}\label{Ec::PSc}I(x_k)\rightarrow c\,,\quad I'(x_k)\rightarrow 0\,,\end{equation}
	has a convergent subsequence in $X$. Let us denote as a $(PS)_c-$sequence, every sequence $(x_k)\subset X$ such that verifies \eqref{Ec::PSc}.
	
	\begin{theorem}[{\cite[Mountain-pass theorem, Brezis and Nirenberg]{BreNir}}]\label{T::2}
		Let $X$ be a Banach space with norm $\left\| \cdot\right\|$, $I\in C^1(X,\mathbb{R})$ and suppose that there exist $r>0$, $\alpha>0$ and $e\in X$ such that $\left\| e\right\| >r$ and
		\begin{enumerate}
			\item $I(x)\geq \alpha$ if $\left\| x\right\| =r$,
			\item $I(e)<0$.
		\end{enumerate}
		Let $c=\inf_{\gamma\in\Gamma}\left\lbrace \max_{t\in[0,1]}I(\gamma(t))\right\rbrace \geq \alpha$, where
		\[\Gamma=\left\lbrace \gamma\in C([0,1],X)\, \mid\, \gamma(0)=0\,,\ \gamma(1)=e\right\rbrace \,.\]
		Then, there exists a $(PS)_c-$sequence for $I$.
		
		Moreover, if $I$ satisfies the $(PS)_c$ condition, then $c$ is a critical value of $I$, that is, there exists $u_0\in X$ such that $I(u_0)=c$ and $I'(u_0)=0$.
			
		\end{theorem}
		
	Let us consider the following norm in $\ell^q$:
	\[\left\| u\right\| _q:=\left( \dfrac{1}{q}\sum_{k\in\mathbb{Z}}V(k)\,\left| u(k)\right| ^q\right) ^{1/q}\,.\]
	
	From the assumption on $V$, we have that it is an equivalent norm to $\left|\cdot \right|_q$, since we have:
	\[\frac{V_0}{q}\left| u\right|_q ^q\leq \left\| u\right\|_q^q \leq \dfrac{V_1}{q}\left| u\right| _q^q\,.\] 	
	
	Now, we have the following result, which can be prove in the same way as \cite[Lemma 2.3]{CaLiTe}:
	\begin{lemma}\label{L::3}
		Suppose that assumptions $(F_1)-(F_3)$ are verified. Then, there exist $\rho>0$, $\alpha>0$ and $e\in \ell^q$ such that $\left\| e\right\| >\rho$ and
		\begin{enumerate}
			\item $J(u)\geq \alpha$ if $\left\| u\right\| =\rho$.
			\item $J(e)<0$.
		\end{enumerate}
	\end{lemma}

	We obtain the following result.
	
	\begin{lemma}\label{L::4}
		Assume that $(F_1)-(F_3)$ are verified. Then, there exists $c>0$ and a $\ell^q-$bounded $(PS)_c$ sequence for $J$.
	\end{lemma}
	\begin{proof}
		From Lemma \ref{L::3} and Theorem \ref{T::2} there exists $(u_m)\subset \ell^q$ a $(PS)_c-$sequence for $J$, i.e, \eqref{Ec::PSc} is verified for $I=J$, where 
		\[c=\inf_{\gamma\in \Gamma}\left\lbrace \max_{t\in[0,1]}J(\gamma(t))\right\rbrace \,,\quad \Gamma=\left\lbrace \gamma\in C([0,1],\ell^q)\,\mid\,\gamma(0)=0\,,\ \gamma(1)=e\right\rbrace \,,\]
		where $e\in\ell^q$ has been introduced in Lemma \ref{L::3}.
		
		Now, we have to prove that the sequence $(u_m)$ is bounded in $\ell^q$. We have,		
		\begin{eqnarray}
		\nonumber\left\langle J'(u_m),u_m\right\rangle &=& \sum_{k\in\mathbb{Z}}\left( \left| \Delta^2\,u_m(k-2)\right|^{p_2}+a\,\left| \Delta\,u_m(k-1)\right| ^{p_1}+V(k)\left| u_m(k)\right| ^q\right) \\\nonumber &&-\sum_{k\in\mathbb{Z}}\lambda \,f(k,u_m(k))\,u_m(k)\,,
		\end{eqnarray}
		Now, using $(F_2)$ and taking into account that $\mu>p_i\geq q>1$ for $i=1,2$, we have:
		\begin{eqnarray}
		\nonumber
		\mu J(u_m)-\left\langle J'(u_m),u_m\right\rangle &=&\left( \dfrac{\mu}{p_2}-1\right) \sum_{k\in\mathbb{Z}} \left| \Delta^2\,u_m(k-2)\right|^{p_2}\\\nonumber&&+a\,\left( \dfrac{\mu}{p_1}-1\right) \sum_{k\in\mathbb{Z}}\left| \Delta\,u_m(k-1)\right| ^{p_1}\\\nonumber &&+\left( \dfrac{\mu}{q}-1\right) \sum_{k\in\mathbb{Z}}V(k)\left| u_m(k)\right| ^q\\\nonumber&&+\lambda  \sum_{k\in\mathbb{Z}}\left( f(k,u_m(k)) u_m(k)-\mu F(k,u_m(k))\right) \\\nonumber
		&\geq& \left( \dfrac{\mu}{q}-1\right) q\left\| u_m\right\| _q^q=\left( \mu-q\right) \left\| u_m\right\| _q^q\,.
		\end{eqnarray}
	Thus, $(u_m)$ is a bounded sequence in $\ell^q$.
	\end{proof}
	
	\begin{remark}
		In order to prove the previous result for problem \eqref{P::h2},we only have to take into account that $\dfrac{\mu}{p_i}-1>0$, for all $i=1,\dots,n$.
	\end{remark}

	Now, we can prove the main result of this part:
	
	\begin{theorem}\label{T::h1}
		Suppose that $a>0$, the function $V\colon \mathbb{Z}\rightarrow \mathbb{R}$ is positive and $T-$periodic and assumptions $(F_1)-(F_3)$ are fulfilled. Then, for $\lambda>0$, problem \eqref{P::h1} has a non trivial homoclinic solution $u\in \ell^q$, which is a critical point of the functional $J\colon \ell^q\rightarrow \mathbb{R}$.
		\end{theorem}
	
	\begin{proof}
		From Lemma \ref{L::4}, we have that the obtained $(PS)_c$ sequence $(u_m)$ is bounded in $\ell^q$, so, $\lim_{\left|k\right| \rightarrow+\infty}\left| u_m(k)\right| =0$. Hence, $\left| u_m(k)\right|$ attains its maximum at a value $k_m\in\mathbb{Z}$.
		
		Let us denote as $j_m$, the unique integer such that $j_m\,T\leq k_m<(j_m+1)\,T$ and let us define the function
		\[w_m(k):=u_m(k+j_m\,T)\,,\]
		this function attains its maximum at $i_m=k_m-j_mT\in[0,T-1]$.

		Taking into account the periodicity of $V$ and $(F_1)$, we have: 
		\[\left\| w_m\right\| _q=\left\| u_m\right\| _q\,,\quad \text{and } J(w_m)=J(u_m)\,.\]
		
		Since $(u_m)$ is bounded in $\ell^q$ it also is $(w_m)$. Then, there exists $w\in \ell^q$ such that $w_m\rightarrow w$ weakly, i.e., $\left\langle w_m,v\right\rangle \rightarrow \left\langle w,v\right\rangle $ for all $v\in\ell^q$.
		
The proof that $w$ is a critical point of the functional $J$ is analogous to  the proof of \cite[Theorem 1.1]{CaLiTe}. Thus, $w$ is a solution of \eqref{P::h1}.

		To finish the proof, we have to see that $w\neq 0$.
		
		Let us assume that $w=0$. In such a case, we have
		\begin{equation}\label{Ec::lim0}\lim_{m\rightarrow+\infty}\left| u_m\right|_{\infty}=\lim_{m\rightarrow+\infty}\left| w_m\right|_{\infty}=\lim_{m\rightarrow+\infty}\max\{\left| w_m\right|\,\mid k\in\mathbb{Z}\}=0\,.\end{equation}
		
		Using $(F_3)$, it is known that for every $\varepsilon>0$, there exists $\delta>0$, such that if $\left| x\right| <\delta$ then, for every $k\in[0,T-1]$, the following inequalities are fulfilled:
		\begin{equation}
		\label{Ec::Ff}
		\left| F(k,x)\right| \leq \varepsilon \left| x\right| ^q\quad \text{and}\quad \left| f(k,x)\,x\right| \leq \varepsilon\,\left| x\right| ^q\,.
		\end{equation}
		
		From \eqref{Ec::lim0}, for every $k\in[0,T-1]$, we can ensure the existence of a positive integer $M_k$, such that if $m>M_k$, then $\left| w_m(k)\right| <\delta$. By the construction of $w_m$, the maximum value of $\left| w_m\right| $ is attained at $i_m\in[0,T-1]$, we conclude that for $m>M=\max\{M_k\ \mid\ k\in[0,T-1]\}$ we have:
		\[\left| w_m(k)\right| \leq\left| w_m(i_m)\right| <\delta\,.\]
		
		Hence, from \eqref{Ec::Ff}, for all $m>M$ we have
		\[\left| F(k,w_m(k))\right| \leq \varepsilon \left| w_m(k)\right| ^q\ \text{and } \left| f(k,w_m(k))\,w_m(k)\right| \leq \varepsilon\,\left| w_m(k)\right| ^q\,,\forall k\in \mathbb{Z}\,.\]
		
		This implies that
		\begin{eqnarray}
		\nonumber
		0\leq qJ(w_m)&=&\dfrac{q}{p_2}\sum_{k\in\mathbb{Z}}\left| \Delta^2w_m(k-2)\right| ^{p_2}+\dfrac{a\,q}{p_1}\sum_{k\in\mathbb{Z}}\left| \Delta w_m(k-1)\right| ^{p_1}\\\nonumber
		&&+\sum_{k\in\mathbb{Z}}V(k)\left| w_m(k)\right| ^q-\lambda q\sum_{k\in\mathbb{Z}}F(k,w_m(k))\,,\\\nonumber
		&\leq& \sum_{k\in\mathbb{Z}}\left( \left| \Delta^2w_m(k-2)\right| ^{p_1}+a\left| \Delta w_m(k-1)\right| ^{p_2}+V(k)\left| w_m(k)\right| ^q\right) \\\nonumber
		&&-\lambda\sum_{k\in\mathbb{Z}}f(k,w_m(k))w_m(k)\\\nonumber &&-\lambda\sum_{k\in\mathbb{Z}}\left( qF(k,w_m(k))-f(k,w_m(k))w_m(k)\right) \,,\\\nonumber
		&\leq &\left\langle J'(w_m),w_m\right\rangle +\lambda\,\sum_{k\in\mathbb{Z}}\left( q\varepsilon\left| w_m(k)\right| ^q+\varepsilon\left| w_m(k)\right| ^q\right) \,,\\\nonumber
		&=&\left\langle J'(w_m),w_m\right\rangle+\lambda\,\varepsilon (q+1)\,\left| w_m\right| _q^q\,,\\\nonumber&\leq& \left\| J'(w_m)\right\| \,\left\| w_m\right\| +\lambda\,\varepsilon \dfrac{q+1}{V_0}\left\| w_m\right\| _q^q\,.
		\end{eqnarray}
		Since $(w_m)\in \ell^q$ is bounded, $\lim_{m\rightarrow+\infty}J'(w_m)=0$ and we can choose $\varepsilon>0$ arbitrarily small, we arrive to a contradiction with the fact that $\lim_{m\rightarrow +\infty}(w_m)=\lim_{m\rightarrow +\infty}(u_m)=c>0$.
	\end{proof}
	
	\begin{remark}
		Taking into account that for every $i=1,\dots,n$, we have that $\dfrac{q}{p_i}\leq 1$, we can prove similarly Theorem \ref{T::h1} for problem \eqref{P::h2}, if we consider $a_i\geq 0$ for all $i=1,\dots,n-1$.
	\end{remark}
	\begin{exemplo}
		Let $r>p_i\geq q>1$ for $i=1,2$ and $b\colon \mathbb{Z}\rightarrow \mathbb{R}$ a positive $T-$periodic function.
		
		Consider $f(k,t)=b(k)\varphi_r(t)$. Let us verify that  such $f$ satisfies $(F_1)-(F_3)$.
		
		\begin{itemize}
			\item[$(F_1)$] Obviously $f$ is continuous as a function of $t$ and $T-$periodic as a function of $k$,
			\item[$(F_2)$] $F(k,t)=b(k)\,\Phi_r(t)$. 
			There exist $r\in \mathbb{R}$, such that $r> p_i\geq q>1$, for $i=1,2$  such that
			\begin{eqnarray}
			\nonumber
			r\,F(k,t)=b(k)\,|t|^r=t\,b(k)\,t\,|t|^{r-2}=t\,f(k,t)\,, && \forall k\in\mathbb{Z}\,,\  t\neq0\,,\\\nonumber
			F(k,t)>0\,,&& \forall k\in\mathbb{Z}\,,\ \forall t>0\,.
			\end{eqnarray}
			\item[$(F_3)$] Since $r>q$, we have:
\[\lim_{|t|\rightarrow 0}\dfrac{f(k,t)}{|t|^{q-1}}=\lim_{|t|\rightarrow 0} b(k)\,t\,|t|^{r-q-1}=0\,.\]
		\end{itemize}
		
		Then, for $V$, $b\colon \mathbb{Z}\rightarrow \mathbb{R}$ two positive $T-$periodic functions the problem
		\begin{eqnarray}
		\nonumber
		\Delta^2\left( \varphi_{p_2}\left( \Delta^2u(k-2)\right) \right) -a\,\Delta\left( \varphi_{p_1}\left( \Delta u(k-1)\right) \right) \\\nonumber +V(k)\,\varphi_q\left( u(k)\right) &=&\lambda \,b(k)\,\varphi_r\left( u(k)\right)\,,\\\nonumber
		\lim_{\left| k\right| \rightarrow +\infty}\left| u(k)\right| &=&0\,,
		\end{eqnarray}
		where $r>p_i\geq q>1$ for $i=1,2$, has a non trivial homoclinic solution for every $\lambda>0$.
	\end{exemplo}
	\section{Boundary value problems}
	In this part we are going to study the existence of multiple solutions for the following boundary value problem:	
	 \begin{eqnarray}
	 	\label{P::d1}
	 	\Delta^2\left( \varphi_{q}\left( \Delta^2u(k-2)\right) \right) -a\Delta\left( \varphi_{q}\left( \Delta u(k-1)\right) \right) \\\nonumber +V(k)\varphi_q\left( u(k)\right) -\lambda f(k,u(k))&=&0,\,k\in[1,T]\\\label{P::cf}
	 u(0)=u(T+1)=\Delta u(-1)=\Delta u(T) &=&0\,,
	 	\end{eqnarray}
	where $T$ is a fixed positive integer, $[1,T]=\{1,2,\dots,T\}$, the difference operators have been introduced in \eqref{Ec::D1}-\eqref{Ec::D2} and $\varphi_q$ has been defined in $\eqref{Ec::Phi}$ for $1<q<+\infty$.
	
	Moreover, we consider $V\colon [1,T]\rightarrow \mathbb{R}$, as a positive function. We can consider it as a restriction to the discrete interval $[1,T]$ of the $T-$periodic function $V$ introduced in the first part of the paper.
	
	 We also denote $V_0=\min\{V(1),\dots,V(T)\}$ and $V_1=\max\{V(1),\dots,V(T)\}$.
	
	Finally, let $f\colon [1,T]\times \mathbb{R}\rightarrow \mathbb{R}$ be a continuous function.
	
	Let $n<T/2$ and $a_i>0$ for $i=1,\dots,n-1$. We are going to indicate how our arguments must be modified in order to obtain the existence of multiple solutions for the following problem:	
	\begin{eqnarray}
	\label{P::d2}
		\Delta^n\left( \varphi_{p_n}\left( \Delta^nu(k-2)\right) \right) +\sum_{i=1}^{n-1}(-1)^ia_i\Delta^{n-i}\left( \varphi_{p_{n-i}}\left( \Delta^{n-i} u(k-1)\right) \right)&& \\\nonumber
		+(-1)^nV(k)\,\varphi_q\left( u(k)\right) +(-1)^{n+1}\lambda f(k,u(k))&=0\,,\\\label{P::cf1} u(0)=\Delta u(-1)=\Delta^2u(-2)=\cdots =\Delta^{n-1}u(1-n) &=0\,,\\\label{P::cf2}
		u(T+1)=\Delta u(T+1)=\Delta^2 u(T+1)=\cdots =\Delta^{n-1}u(T+1)&=0\,.
		\end{eqnarray}
	
	As in the first part, we are going to obtain the existence result by means of variational methods.
	
	In order to obtain our variational approach, we consider the following $T-$dimensional Banach space:	
 \begin{equation}\label{Ec::esp} 
	X=\left\lbrace u\colon [-1,T+2]\rightarrow \mathbb{R}\ \mid\ u(0)=u(T+1)=\Delta u(-1)=\Delta u (T+1)=0\right\rbrace\,, \end{equation}
coupled with the following norm
\[\left\| u\right\| _X=\left( \sum_{k=1}^{T+2}\left| \Delta^2(k-2)\right| ^q+a\sum_{k=1}^{T+1}\left| \Delta u(k-1)\right| ^q+\sum_{k=1}^TV(k)\,\left| u(k)\right| ^q\right) ^{1/q}\,.\]

\begin{remark}
In order to study problem \eqref{P::d2}--\eqref{P::cf2}, we consider the following Banach space
		\begin{eqnarray}
			\nonumber X_n&=&\left\lbrace u\colon [1-n,T+n]\rightarrow \mathbb{R}\ \mid\ u(0)=u(T+1)=\cdots=\Delta^{n-1} u(1-n)\right. \\\nonumber &&\left. =\Delta^{n-1} u (T+n)=0\right\rbrace\,, \end{eqnarray}
		coupled with the following norm
				\[\left\| u\right\| _{X_n}=\left( \sum_{k=1}^{T+n}\left| \Delta^n(k-n)\right| ^q+\sum_{i=1}^{n-1}a_i\sum_{k=1}^{T+n-i}\left| \Delta^{n-i} u(k-(n-i))\right| ^q+\sum_{k=1}^TV(k)\,\left| u(k)\right| ^q\right) ^{1/q}\,.\]	
\end{remark}

We have the following result, in terms of the norm $\left\| \cdot\right\|_X $.

\begin{lemma}\label{L::1d}
	For every $u\in X$, the following inequality holds:
	\[\max_{k\in[1,T]}\left| u(k)\right| \leq \rho \left\| u\right\| _X\,,\]
	where
	\begin{equation}\label{Ec::rho}\rho=\dfrac{(T+1)\,(T+2)^{\frac{q-1}{q}}}{\left( 4^q+2^q\,a\,(T+1)\,(T+2)^{q-1}+V_0\,(T+1)^q\,(T+2)^{q-1}\right) ^{1/q}}\,.\end{equation}
\end{lemma}

\begin{proof}
	First, we have:
	
	\begin{equation}\label{Ec::n1}\max_{k\in[1,T]}\left| u(k)\right| ^q\leq \sum_{k=1}^T\left| u(k)\right| ^q\leq\dfrac{1}{V_0} \sum_{k=1}^TV(k)\,\left| u(k)\right| ^q\,.\end{equation}
	
	Secondly, taking into account the boundary conditions \eqref{P::cf}, for every $u\in X$ and all $j=1,\dots,T$, we have	
	\begin{eqnarray}
	\nonumber
	\sum_{k=1}^{T+1}\left| \Delta u(k-1)\right| &=&\sum_{k=1}^j\left| u(k)-u(k-1)\right| +\sum_{k=j+1}^{T+1}\left| u(k-1)-u(k)\right|\\\nonumber &\geq& \sum_{k=1}^j\left( \left| u(k)\right| -\left| u(k-1)\right|\right)  +\sum_{k=j+1}^{T+1}\left( \left| u(k-1)\right| -\left| u(k)\right|\right)\\\label{Ec::n2} &=&2\left| u(j)\right| -\left| u(0)\right| -\left| u(T+1)\right| =2\left| u(j)\right| \,. 
	\end{eqnarray}
	
	Analogously, for every $u\in X$ and all $j=1,\dots,T$, we obtain
	
	\begin{equation}\label{Ec::n3}
	\sum_{k=1}^{T+2}\left| \Delta^2u(k-2)\right| \geq 2\left| \Delta u(j-1)\right| \,.
	\end{equation}
	
	Now, combining \eqref{Ec::n2} with the discrete H\"{o}lder inequality, we have
{\small 	\begin{equation}\label{Ec::n4}
	\max_{k\in[1,T]}\left| u(k)\right| \leq \dfrac{1}{2}\sum_{k=1}^{T+1}\left| \Delta u(k-1)\right| \leq \dfrac{1}{2}(T+1)^{\dfrac{q-1}{q}}\left( \sum_{k=1}^{T+1}\left| \Delta u(k-1)\right| ^q\right) ^{1/q}\,.\end{equation}}
	
	Using \eqref{Ec::n2} and \eqref{Ec::n3}, we have:
	{\small 	\begin{eqnarray}\nonumber
		\max_{k\in[1,T]}\left| u(k)\right|& \leq& \dfrac{1}{2}\sum_{k=1}^{T+1}\left| \Delta u(k-1)\right| \leq \dfrac{1}{2}(T+1)\max_{k\in[1,T+1]}\left| \Delta u(k-1)\right| ^q\\\label{Ec::n5}\\\nonumber
	&\leq& \dfrac{T+1}{4}\sum_{k=1}^{T+2}\left| \Delta^2u(k-2)\right| \leq \dfrac{(T+1)(T+2)^{\frac{q-1}{q}}}{4}\left( \sum_{k=1}^{T+2}\left| \Delta^2 u(k-2)\right| ^q\right) ^{1/q}	\,.\end{eqnarray}}
	
	Thus, from \eqref{Ec::n1}, \eqref{Ec::n4} and \eqref{Ec::n5}, we obtain
	\begin{equation}
	\left(\dfrac{4^q}{(T+1)^q(T+2)^{q-1}}+ \dfrac{a\,2^q}{(T-1)^{q-1}}+V_0\right) \max_{k\in[1,T]}\left| u(k)\right|^q \leq \left\| u\right\| _X^q\,,
	\end{equation}
	and the result is proved, taking into account that \[\dfrac{4^q}{(T+1)^q(T+2)^{q-1}}+ \dfrac{a\,2^q}{(T-1)^{q-1}}+V_0=\dfrac{1}{\rho^q}\,.\]
\end{proof}

\begin{remark}
	Using induction arguments and the discrete H\"{o}lder inequality, it can be proved that for each $i=1,\dots,n$
	\[\sum_{k=1}^{T+i}\left|\Delta^iu(k-i)\right| ^q\geq \dfrac{2^{q\,i}}{(T+i)^{q-1}\prod_{j=1}^{i-1}(T+j)^q}\max_{k\in[1,T]}\left| u(k)\right| ^q\,.\]
	
	Thus, Lemma \ref{L::1d} remains true for problem \eqref{P::d2}--\eqref{P::cf2} with
	\[\rho=\dfrac{(T+n)^{\frac{q-1}{q}}\prod_{j=1}^{n-1}(T+j)}{\left( 2^{q n}+\sum_{i=1}^{n-1}a_i2^{q(n-i)}(T+n-i)\prod_{j=n-i+1}^{n-1}(T+j)^q+V_0(T+n)^{\frac{q-1}{q}}\prod_{j=1}^{n-1}(T+j)\right) ^{(1/q)}}.\]	
\end{remark}

	Now, let us consider $J_1\colon X\longrightarrow \mathbb{R}$, the  functional $J(u)=\Phi_1(u)-\lambda\sum_{k=1}^{T} F(k,u(k))$,
where $F(k,t)=\int_0^tf(k,s)\,ds$ for every $(k,t)\in [1,T]\times \mathbb{R}$ and
	\[\Phi_1(u)=\sum_{k=1}^{T+2} \phi_{q}\left( \Delta^2\,u(k-2)\right) +a\sum_{k=1}^{T+1}\phi_{q}\left(\Delta\,u(k-1)\right) +\sum_{k=1}^{T}V(k)\,\phi_q(u(k)) \,.\]
	
	\begin{remark}
		For Problem \eqref{P::d2}--\eqref{P::cf2}, we have
			\[\Phi_1(u)=\sum_{k=1}^{T+n} \phi_{q}\left( \Delta^n\,u(k-n)\right) +\sum_{i=1}^{n-1}a_i\sum_{k=1}^{T+n-i}\phi_{q}\left(\Delta^{n-i}\,u(k-(n-i))\right) +\sum_{k=1}^{T}V(k)\,\phi_q(u(k)) \,.\]
	\end{remark}
	
	\begin{remark}
		\label{R::1}
	Realize that, in both cases, $\Phi_1(u)=\dfrac{\left\| u\right\| _X^q}{q}$, hence trivially $\Phi_1$ is coercive, that is, $\lim_{\left\| u\right\| _X\rightarrow+\infty}\Phi_1(u)=+\infty$.
		\end{remark}
		
	We have the analogous result for this case to Lemma \ref{L::2}
	
		\begin{lemma}\label{L::2d}
			The functional $J\colon X\rightarrow \mathbb{R}$ is  $C^1-$differentiable and its critical points are solutions of \eqref{P::d1}.
		\end{lemma}
		
		\begin{proof}
				For all $v\in X$, we have:				
				\begin{eqnarray}\nonumber
				\left\langle J_1'(u),v\right\rangle &=& \sum_{k=1}^{T+2} \varphi_{q}\left( \Delta^2u(k-2)\right) \,\Delta^2v(k-2)+a\sum_{k=1}^{T+1}\varphi_{q}\left( \Delta u(k-1)\right) \Delta v(k-1)
				\\\label{Ec::DJ} \\\nonumber &&+\sum_{k=1}^T V(k)\,\varphi_q(u(k))v(k)-\lambda\sum_{k=1}^T f(k,u(k))\,v(k) \,.
				\end{eqnarray}
				
				Now, for every $u$, $v\in X$ we have:
		{\small 	\begin{eqnarray}
			\nonumber \sum_{k=1}^{T+1}\varphi_q\left( \Delta u(k-1)\right) \Delta v(k-1)&=&\sum_{k=1}^{T+1}\varphi_q\left( \Delta u(k-1)\right) v(k)-\sum_{k=0}^{T}\varphi_q\left( \Delta u(k)\right) v(k)\\\nonumber
			&=&-\sum_{k=1}^T\Delta \left( \varphi_q\left( \Delta u(k-1)\right) \right) v(k)\\\nonumber
			&&+\varphi_q\left( \Delta u(T)\right) v(T+1)-\varphi_q\left( \Delta u(0)\right) v(0)\\\nonumber &=&-\sum_{k=1}^T\Delta \left( \varphi_q\left( \Delta u(k-1)\right) \right) v(k)\,,
			\end{eqnarray}}
			and
				{\small \begin{eqnarray}
				\nonumber \sum_{k=1}^{T+2}\varphi_q\left( \Delta^2 u(k-2)\right) \Delta^2 v(k-2)&=&\sum_{k=1}^{T+2}\varphi_q\left( \Delta^2 u(k-2)\right) \Delta v(k-1)\\\nonumber&&-\sum_{k=0}^{T+1}\varphi_q\left( \Delta^2 u(k-1)\right) \Delta v(k-1)\\\nonumber
				&=&-\sum_{k=1}^{T+1}\Delta \left( \varphi_q\left( \Delta^2 u(k-2)\right) \right) \Delta v(k-1)\\\nonumber
				&&+\varphi_q\left( \Delta^2 u(T)\right) \Delta v(T+1)-\varphi_q\left( \Delta^2 u(-1)\right) \Delta v(-1)\\\nonumber &=&-\sum_{k=1}^{T+1}\Delta \left( \varphi_q\left( \Delta^2 u(k-2)\right) \right) \Delta v(k-1)\\\nonumber
				&=&-\sum_{k=1}^{T+1}\Delta\left( \varphi_q\left( \Delta^2 u(k-2)\right)\right)  v(k)\\\nonumber &&+\sum_{k=0}^{T}\Delta \left( \varphi_q\left( \Delta^2 u(k-1)\right)\right)  v(k)\\\nonumber
				&=&\sum_{k=1}^T\Delta^2 \left( \varphi_q\left( \Delta^2 u(k-2)\right) \right) v(k)\\\nonumber
				&&-\Delta\left( \varphi_q\left( \Delta^2 u(T-1)\right)\right)  v(T+1)+\Delta \left( \varphi_q\left( \Delta^2 u(-1)\right)\right)  v(0)\\\nonumber &=&\sum_{k=1}^T\Delta^2 \left( \varphi_q\left( \Delta^2 u(k-2)\right) \right) v(k)\,.
				\end{eqnarray}}
				
				Thus, for all $u$, $v\in X$ we can write \eqref{Ec::DJ} in an equivalent way as follows :
					\begin{eqnarray}\nonumber
					\left\langle J_1'(u),v\right\rangle &=& \sum_{k=1}^{T}\Delta^2\left(  \varphi_{q}\left( \Delta^2u(k-2)\right)\right)  \,v(k)-a\sum_{k=1}^{T+1}\Delta \left( \varphi_{q}\left( \Delta u(k-1)\right)\right)  v(k) \\\nonumber &&+\sum_{k=1}^T V(k)\,\varphi_q(u(k))v(k)-\lambda\sum_{k=1}^T f(k,u(k))\,v(k) \,.
					\end{eqnarray}
					
					From the arbitrariness of $v\in X$, we conclude that if $u$ is a critical point of $J_1$, then it is a solution of \eqref{P::d1}-\eqref{P::cf} and the result is proved.
		\end{proof}
		
		\begin{remark}
			In order to prove Lemma \ref{L::2d} for problem \eqref{P::d2}--\eqref{P::cf2}, we see, using an induction hypothesis, that for every $i=1,\dots,n$ the following equality is fulfilled for all $u$, $v\in X$.
				\[\sum_{k=1}^{T+i}\varphi_{q}\left( \Delta^{i}u(k-i)\right) \Delta^{i}v(k-i)=(-1)^{i}\sum_{k=1}^T\Delta^{i}\left(\varphi_{q}\left( \Delta^{i}u(k-i)\right) \right) v(k)\,.\]
			
			Such equality allows us to prove, using the arguments of Lemma \ref{L::2d} that every critical point of $J_1$ is a solution of problem \eqref{P::d2}--\eqref{P::cf2}.
			\end{remark}
			
		 Let us denote $\Psi_1(u)=-\sum_{k=1}^TF(k,u(k))$.
		 
		 Let $E$ be a finite dimensional Banach space and consider $J\colon E \rightarrow \mathbb{R}$ the functional $J(u)=\Phi (u)+\lambda\,\Psi (u)$,		 where $\Phi$, $\Psi\colon X\rightarrow \mathbb{R}$ are of class $C^1(E)$ and $\Phi$ is coercive.
		 		 
		 \begin{remark}
		 	Realize that $J_1(u)=\Phi_1(u)+\lambda\Psi_1(u)$ satisfies this condition.
		 \end{remark}
		 
		 Furthermore, for every $r>\inf_E\Phi$, let us define:
		 \begin{eqnarray}
		 \nonumber \psi_1(r)&:=& \inf_{u\in \Phi^{-1}\left( \left( -\infty,r\right) \right) }\dfrac{\Psi(u)-\inf_{\Phi^{-1}\left( \left( -\infty,r\right] \right)} \Psi}{r-\Phi(u)}\,,\\
		  \nonumber \psi_2(r)&:=& \inf_{u\in \Phi^{-1}\left( \left( -\infty,r\right) \right) } \sup_{v\in \Phi^{-1}\left( \left[r,+\infty\right) \right) }\dfrac{\Psi(u)-\Psi(v)}{\Phi(v)-\Phi(u)}\,,
		 \end{eqnarray}
		 
		 Now, we are under conditions of applying \cite[Theorem 2.1]{CanGio} to our problem. 
		 
		 \begin{theorem}[{\cite[Theorem 2.1]{CanGio}}] \label{T::2.1}
		 Assume that:
		 \begin{enumerate}
		 	\item there exists $r>\inf_E\Phi$ such that $\psi_1(r)<\psi_2(r)$.
		 	\item for each $\lambda\in\left( \dfrac{1}{\psi_2(r)},\dfrac{1}{\psi_1(r)}\right)$ we have $\lim_{\left\| u\right\|_E\rightarrow+\infty}J(u)=+\infty$.
		 \end{enumerate}
		 Then, for each $\lambda\in\left( \dfrac{1}{\psi_2(r)},\dfrac{1}{\psi_1(r)}\right)$, $J$ has at least three critical points.
		\end{theorem}

		Let us define, for $c$, $d>0$,
		\begin{eqnarray}
		\nonumber
		\varTheta(c)&:=&\dfrac{\sum_{k=1}^T\sup_{\left| s\right| \leq c}F(k,s)}{c^q}\,,\\\nonumber
	\varLambda(d)&:=&\dfrac{\sum_{k=1}^T\left( F(k,d)-\sup_{\left| s\right| \leq c}F(k,s)\right) }{d^q}\,.
		\end{eqnarray}
		
		Now, we can state the main result of this section:
		
		\begin{theorem}\label{T::d}
			Assume that there exist four positive constants $b$, $c$, $d$, and $p$, such that $c<d$ and $p<q$ verifying:
			\begin{itemize}
				\item[$(d_1)$] $\varTheta(c)<\dfrac{\varLambda(d)}{\left( 4+2\,a+T\,V_1\right) \,\rho^q}$, where $\rho$ has been introduced in \eqref{Ec::rho}.
				\item[$(d_2)$] $F(k,t)\leq b\,\left( 1+\left| t\right| ^p\right) $ for all $(k,t)\in[1,T]\times \mathbb{R}$.
			\end{itemize}
			
			Then, for every $\lambda\in \left( \dfrac{4+2\,a+T\,V_1}{q\,\varLambda(d)},\dfrac{1}{q\,\rho^q\,\varTheta(c)}\right)$, the Problem \eqref{P::d1}-\eqref{P::cf} admits at least three solutions which are critical points of $J_1$.
		\end{theorem}
		
		\begin{proof}
			We just need to find $r>\inf_X \Phi$ such that the hypothesis of Theorem \ref{T::2.1} are verified.
			
			Let us consider:
			\[r=\dfrac{c^q}{q\,\rho^q}\,.\]
			
			Taking into account the relationship between $\Phi_1$ and the norm $\left\| \cdot \right\| _X$, we have:
			\begin{eqnarray}
			\nonumber \psi_1(r)&=&\inf_{\left\| u\right\|_X <\left( q\,r\right) ^{1/q} }\dfrac{\Psi_1(u)-\inf_{\left\| u\right\|_X \leq\left( q\,r\right) ^{1/q}} \Psi_1(u)}{r-\Phi_1(u)}\leq \dfrac{-\inf_{\left\| u\right\| _X\leq \left( q\,r\right)^{1/q}}\Psi(u)}{r}\\\nonumber
				 &=&\dfrac{\sup_{\left\| u\right\| _X\leq (q\,r)^{1/q}}\sum_{k=1}^TF(k,u(k))}{r}\,.
				\end{eqnarray}
				
				Now, from Lemma \ref{L::1d}, if $\left\| u\right\| _X\leq (q\,r)^{1/q}$, then for all $k\in[1,T]$:
				\[|u(k)|\leq \rho \left( q\,r\right) ^{1/q}=c\,,\]
				thus,
			\begin{equation*}
			 \psi_1(r)\leq 	\dfrac{\sum_{k=1}^T\sup_{|u(k)|\leq c}F(k,u(k))}{r} =q\,\rho^q\,\varTheta(c)\,.
			\end{equation*}
			
			Now, let us see that $c^q<\rho^q\,\left( 4+2\,a+V_0\right) \,d^p$. In order to do that, let us choose $k^*\in[1,T]$, such that $V(k^*)=V_0$, and consider:
			\[v_c(k):=\left\lbrace \begin{array}{cc}
			c&\text{if } k=k^*\,,\\
			0&\text{if } k\neq k^*\,,\end{array}\right. \]
			from Lemma \ref{L::1d}, since $c<d$, we have:
			\[c^q\leq \rho^q\,\left\| v_c\right\| _X^q=\rho^q\,(4+2\,a+V_0)\,c^q<\rho^q\,(4+2\,a+V_0)\,d^q\,.\]
			
			Now, consider $v_d\in X$, such that:
				\[v_d(k):=\left\lbrace \begin{array}{cc}
				d&\text{if } k\in[1,T]\,,\\
				0&\text{otherwise\,.} \end{array}\right. \]
				
				We have
				\[\left\| v_d\right\| _X^q=\left( 4+2\,a+\sum_{k=1}^TV(k)\right) d^q\geq \left( 4+2\,a+V_0\right) d^q>\left( \dfrac{c}{\rho}\right) ^q=q\,r\,.\]
				
				Hence,				
				\begin{eqnarray}
				\nonumber\psi_2(r)&=&\inf_{\left\| u\right\| _X<(q\,r)^{1/q}}\sup_{\left\| v\right\|_X\geq (q\,r)^{1/q}}\dfrac{\Psi_1(u)-\Psi_1(v)}{\Phi_1(v)-\Phi_1(u)}\geq \inf_{\left\| u\right\| _X<(q\,r)^{1/q}}\dfrac{\Psi_1(u)-\Psi_1(v_d)}{\Phi_1(v_d)-\Phi_1(u)}\\\nonumber
				&=&q\,\inf_{\left\| u\right\| _X<(q\,r)^{1/q}}\dfrac{\sum_{k=1}^T F(k,d)-\sum_{k=1}^T F(k,u(k))}{\left( 4+2\,a+\sum_{k=1}^TV(k)\right) d^p-\left\|u\right\| _X^q }\\\nonumber&\geq& q\,\dfrac{\sum_{k=1}^T F(k,d)-\sum_{k=1}^T \sup_{|s|\leq c}F(k,u(k))}{\left( 4+2\,a+\sum_{k=1}^TV(k)\right) d^p }\,.
				\end{eqnarray}
				Now, taking into account that $\sum_{k=1}^TV(k)\leq T\,V_1$, we have
				\[\psi_2(r)\geq q\dfrac{\varLambda(d)}{4+2\,a+T\,V_1}\,,\]
				thus, from condition $(d_1)$, we conclude:
				\[\psi_2(r)>q\,\rho^q\varTheta(c)\geq \psi_1(r)\,.\]
				
				For another hand, from condition $(d_2)$, we have:				
				\[J_1(u)=\dfrac{\left\| u\right\| _X^q}{q}-\lambda\sum_{k=1}^T F(k,u(k))\geq \dfrac{\left\| u\right\| _X^q}{q}-\lambda\,b\sum_{k=1}^T (1+|u(k)|^s)\,,\]
				using again Lemma \ref{L::1d}, we conclude
				\[J_1(u)\geq \dfrac{\left\| u\right\| _X^q}{q}-\lambda\,T\,b-\lambda\,T\,\rho^p\,\left\| u\right\| _X^p\,,\]
				which, since $p<q$, ensures that $\lim_{\left\| u\right\| _X\rightarrow+\infty}J_1(u)=+\infty$.
				
				Therefore, by applying Theorem \ref{T::2.1} the result is proved.
		\end{proof}
		
		\begin{remark}
			We can state Theorem \ref{T::d} for problem \eqref{P::d2}--\eqref{P::cf2}, applying similar arguments as follows:
			\begin{theorem}
					Assume that there exist four positive constants $b$, $c$, $d$, and $p$, such that $c<d$ and $p<q$ verifying:
					\begin{itemize}
						\item[$(d_1)$] $\varTheta(c)<\dfrac{\varLambda(d)}{\left( 2^n+\sum_{i=1}^{n-1}2^{n-i}a_i+T\,V_1\right) \,\rho^q}$, where $\rho$ has been introduced in \eqref{Ec::rho}.
						\item[$(d_2)$] $F(k,t)\leq b\,\left( 1+\left| t\right| ^p\right) $ for all $(k,t)\in[1,T]\times \mathbb{R}$.
					\end{itemize}
					
					Then, for every $\lambda\in \left( \dfrac{2^n+\sum_{i=1}^{n-1}2^{n-i}a_i+T\,V_1}{q\,\varLambda(d)},\dfrac{1}{q\,\rho^q\,\varTheta(c)}\right)$, the Problem \eqref{P::d2}--\eqref{P::cf2} admits at least three solutions which are critical points of $J_1$.
			\end{theorem}
		\end{remark}
		
		\begin{exemplo}
			Let $T=8$ and $V(k)=k$ for each $k\in[1,T]$. Then, in this case $V_0=1$ and $V_1=8$.
			
			 Moreover, consider $f(k,t)=k^2\,g(t)$, where
			 \[g(t)=\left\lbrace \begin{array}{cc}
			 e^t&t\leq14\\
			 e^{14}&t>14\,,\end{array}\right. \]
			 then, $F(k,t)=k^2\,G(t)$, where
			 	 \[G(t)=\left\lbrace \begin{array}{cc}
			 	 e^t&t\leq14\\
			 	 e^{14}(t-13)&t>14\,.\end{array}\right. \]
			 	 
			 	 So, we can see that $F(k,t)\leq64\, e^{14}(1+|t|^p)$ for all $p>1$.
			 	 
			 	Let us choose $c=1$, $d=14$ and $q=3$, we have:
			 	 \begin{eqnarray}
			 	 	\nonumber
			 	 	\varTheta(1)&:=&\dfrac{\sum_{k=1}^8\sup_{\left| s\right| \leq 1}F(k,s)}{1^3}=204\, e\,,\\\nonumber
			 	 	\varLambda(14)&:=&\dfrac{\sum_{k=1}^8\left( F(k,14)-\sup_{\left| s\right| \leq 1}F(k,s)\right) }{14^3}=\dfrac{51}{686}(e^{14}-e)\,.
			 	 	\end{eqnarray}
			 	 	
			 	 Now,  consider the problem:
			 	  \begin{eqnarray}
			 	  \label{Ex::2}
			 	  \Delta^2\left( \varphi_{3}\left( \Delta^2u(k-2)\right) \right) -10\Delta\left( \varphi_{3}\left( \Delta u(k-1)\right) \right) \\\nonumber +V(k)\varphi_3\left( u(k)\right) -\lambda f(k,u(k))&=&0,\,k\in[1,8]\,, \end{eqnarray}
			 	  coupled with the boundary conditions \eqref{P::cf}.
			 	  
			 	  In this case, $\rho^3=\dfrac{18225}{36241}$.
			 	  
			 	  Moreover, we have
			 	  \[\dfrac{\varLambda(14)}{\left( 4+20+8^2\right) \,\rho^3}=\dfrac{616097}{366735600}(e^{14}-e)\approxeq 2020.31>204 e\approxeq 554.53\,,\]
			 	  then, we can apply Theorem \ref{T::d} to conclude that for each $\lambda\in \left( \dfrac{19208}{51(e^{14}-e)},\dfrac{36241}{11153700\,e}\right) $ the Problem \eqref{Ex::2} has at least three solutions.
			\end{exemplo}

\end{document}